\documentclass[12pt]{amsart}
\usepackage{amsmath}
\usepackage{amsfonts,amsthm,amsopn,cite,mathrsfs,amssymb}
\usepackage{epsfig,verbatim}
\usepackage{graphicx}
\usepackage{subfigure}
\usepackage{xcolor, hyperref} 
\usepackage{enumerate}

\newcommand{\tfa}{time-frequency analysis}

\newcommand{\fif}{if and only if}
\newcommand{\tfs}{time-frequency shift}

\newtheorem{tm}{Theorem}[section]    
\newtheorem{lemma}[tm]{Lemma}
\newtheorem{prop}[tm]{Proposition}
\newtheorem{cor}[tm]{Corollary}

\theoremstyle{definition}
\newtheorem{definition}{Definition}[section]

\newcommand{\beqa}{\begin{eqnarray*}}
\newcommand{\eeqa}{\end{eqnarray*}}

\DeclareMathOperator*{\supp}{supp}

\newcommand{\field}[1]{\mathbb{#1}}
\newcommand{\bR}{\field{R}}        
\newcommand{\bN}{\field{N}}        
\newcommand{\bC}{\field{C}}        
\def\cF{\mathcal{ F}}              

\def\cH{\mathcal{ H}}

\def\cG{\mathcal{ G}}

\def\cO{\mathcal{ O}}

\def\cP{\mathcal{ P}}

\def\rd{\bR^d}
\def\cd{\bC^d}

\def\<{\left<}
\def\>{\right>}

\def\inv{^{-1}}

\def\mv1{M_v^1}

\newcommand{\mz}{Marcinkie\-wicz-\-Zyg\-mund}

\newcommand{\pln}{\cP _n}
\newcommand{\fc}{_{\cF ^2}}

\begin{document}

\begin{abstract}
We study  the relation between Marcinkiewicz-Zygmund families  for 
polynomials in a weighted $L^2$-space and sampling  theorems  for entire 
functions in the Fock space and the dual relation between 
uniform interpolating families  for polynomials and interpolating
sequences.  As a consequence we obtain a description
of signal subspaces spanned by Hermite functions
by means of Gabor frames.  
\end{abstract}

\title[Marcinkiewicz-Zygmund Inequalities]{Marcinkiewicz-Zygmund
  Inequalities for Polynomials in Fock Space}
\author{Karlheinz Gr\"ochenig}
\address{Faculty of Mathematics \\
University of Vienna \\
Oskar-Morgenstern-Platz 1 \\
A-1090 Vienna, Austria}
\email{karlheinz.groechenig@univie.ac.at}

\author{Joaquim Ortega-Cerd\`a}
\address{Departament de Matem\`atiques i Inform\`atica \\
Universitat de  Barcelona \\
Gran Via de les Corts Catalanes, 585 \\
08007, Barcelona, Spain}
\email{jortega@ub.edu}

\date{}
\thanks{K.\ G.\ was
  supported in part by the  project P31887-N32  of the
Austrian Science Fund (FWF). J.O.C.\ has been partially supported by
the Generalitat de Catalunya (grant 2017 SGR 359) and the Spanish
Ministerio de Ciencia,  Innovaci\'on y Universidades (project
MTM2017-83499-P)}
 \maketitle

\section{Introduction}

We study sampling and interpolation  in Fock space and the relation to
sampling and interpolation of polynomials. 
The Fock space $\cF ^2$ consists of all entire  functions with finite norm
\begin{equation}
  \label{eq:1}
  \|f\|_{\cF^2} = \Big(  \int _{\bC } |f(z)|^2 \, e^{-\pi |z|^2} dm(z)
  \Big)^{1/2} \, , 
\end{equation}
where $dm(z) = dx dy $ is the Lebesgue measure on $\bC \simeq \bR ^2$. 

We denote by $\pln$ the holomorphic polynomials of degree at most $n$.
A sequence of (finite) subsets $\Lambda _n \subseteq \bC $ is called a \mz\
family for the $\pln $  in  
Fock space $\cF ^2$, if there
exist constants $A,B>0$,  such that for all $n$ large, $n\geq n_0$, 
\begin{equation}
  \label{eq:k1}
\qquad   A \|p\|_{\fc }^2 \leq \sum _{\lambda \in \Lambda _n} \frac{|p(\lambda
    )|^2}{k_n(\lambda,\lambda )} \leq B \|p\|_{\fc }^2  \qquad \text{
    for all } p\in \pln   \, .
\end{equation}
Here $k_n$ is the reproducing kernel of $\pln$, when  endowed with the  inner
product inherited from $\cF ^2$. 

This notion  corresponds to the standard definition of sampling in a reproducing kernel 
Hilbert space $\cH$. Let $k_\lambda (z) = k(z,\lambda )$ be the
reproducing kernel of $\cH $, i.e., $f(\lambda ) = \langle f,
k_\lambda \rangle _{\cH }$ at the point $\lambda $. Then a sequence
$\Lambda $ is a sampling set for $\cH 
$, if the  normalized reproducing kernels  
$\Big\{\frac{k(z,\lambda)}{\sqrt{k(\lambda,\lambda)}} : \lambda\in 
\Lambda  \Big\}$  constitute a frame for $\cH$. Equivalently, the sampling
inequality $A\|f\|_\cH ^2 \leq \sum _{\lambda \in \Lambda } |f(\lambda
)|^2 k(\lambda , \lambda )^{-1} \leq  B \|f\|_{\cH }^2$ holds for all
$f\in \cH $. 

 In the Fock space $\cF ^2$  the reproducing kernel  is $k(z,w) =
 e^{\pi z\overline{w}}$, and  a sequence 
$\Lambda \subseteq \bC$ is sampling in $\cF ^2$,  if and only if
\[
   A\|f\|_{\cF^2}\le  \sum _{\lambda \in \Lambda } 
   |f(\lambda)|^2 e^{-\pi |\lambda|^2} \leq B \|f\|_{\fc }^2  \qquad \text{
    for all } f\in \cF^2   \, .
\]

In this article we compare the notion of Marcinkiewicz-Zygmund families for $\pln$ 
with sampling sequences for the Fock space $\cF^2$. We will see that 
both notions 
are intimately connected. Roughly speaking, suitable finite sections
of a sampling set for $\cF ^2$ yield a \mz\ family for the
polynomials $\pln $ in $\cF ^2$, and suitable limits of a  \mz\ family
yield a sampling set for $\cF ^2 $.

A precise formulation is contained in our main  result. (See Section~5
for an explanation of weak limits.) 
  \begin{tm} \label{tmintro1}
(i)   Assume that $\Lambda \subseteq \bC $ is a sampling set for $\cF ^2
  $.  For $\tau >0$ set $\rho _n$ such that $\pi \rho _n^2 = n +
  \sqrt{n} \tau  $ and let $B_{\rho _n}$ be the centered disk of radius
  $\rho _n$. Then for  $\tau >0$ large enough,   the sets $\Lambda _n = 
  \Lambda \cap  B_{\rho_n} $  form a \mz\ family for $\pln $ in
  $\cF ^2$.

  (ii) Conversely, every weak limit of a \mz\ family $(\Lambda _n)$
  for $\pln $ in $\cF ^2 $ is a sampling set for $\cF ^2 $.  
\end{tm}
A dual result establishes a similar relationship between interpolating
sets for $\cF^2 $ and uniformly interpolating sets for $\pln $. A set
$\Lambda $ is interpolating in $\cF ^2$, if for every sequence 
$(a_\lambda)_{\lambda \in \Lambda}\in \ell^2(\Lambda)$, there exists  $f\in \cF 
^2 $ such that $f(\lambda
) e^{-\pi |\lambda |^2/2} = a_\lambda $ for all $\lambda \in \Lambda $. Of course, for
polynomials of degree $n$ every set of $n+1$  points is
interpolating.  In analogy to the  definition of \mz\ families we call a family
of (finite) subsets $\Lambda _n\subseteq \bC $ a uniform interpolating
family, if there exists a constant $A >0$, such that for every sequence $a
= (a_\lambda)_{\lambda \in \Lambda _n } \in \ell
^2(\Lambda _n )$ there exists a polynomial $p\in \pln $ such that $p(\lambda
) k_n(\lambda,\lambda )^{-1/2} = a_\lambda $ for  $\lambda \in \Lambda _n$ with norm control
$\|p\|\fc ^2 \leq A \|a\|_2^2$.  
\begin{tm} \label{tmintro2}
(i)     Assume that $\Lambda \subseteq \bC $ is a  set of interpolation for $\cF ^2
    $.  For $\tau >0$ and $\epsilon >0$ set $\rho _n$, such that $\pi \rho _n^2 =
  n-\sqrt{n} (\sqrt{2\log n} + \tau) $. Then for every   $\tau >0$ large 
enough,   the sets $\Lambda _n = 
  \Lambda \cap  B_{\rho_n} $  form a uniform interpolating family for $\pln $ in
  $\cF ^2$.

  (ii) Conversely, every  weak limit of  a uniform interpolating
  family  $(\Lambda
  _n)$ is a set of interpolation for $\cF ^2 $. 
  \end{tm}

\mz\ families and uniform interpolating families arise several areas
of analysis. They can be understood as finite-dimensional
approximations of sampling theorems in reproducing kernel Hilbert
spaces. Theorem~\ref{tmintro1}(ii)  shows that a \mz\ family can be used to
prove a sampling theorem in an infinite dimensional space. In
approximation theory, a \mz\ family for a sequence of nested subspaces
gives rise to a sequence of quadrature rules and function
approximation from point evaluations, see~\cite{Gro20} for this
aspect. Random constructions of \mz\ families occur in the theory of
deterministic point processes, e.g., \cite{Ameur13,AHM10,AR20,AC21}. In the recent advances
in complexity theory and data analysis \mz\ families are implicit in
the discretization of norms. For a nice survey see~\cite{KKLT}. 

This work has several predecessors in different contexts. In
\cite{GOC21} we  have studied the analogous problem in the Bergman 
space and in the Hardy space in the unit disk. Indeed, our proof
strategy for Theorem~\ref{tmintro1} is taken from ~\cite{GOC21}.  Whereas  in Bergman space
the results can be formulated similarly to Theorem~\ref{tmintro1},
the situation is Hardy space is rather  different and the construction
of \mz\ families needed to be  based on different
principles. In~\cite{Gro99}  a sampling theorem for bandlimited functions
was derived  via \mz\ families for trigonometric polynomials. The set-up of~\cite{LOC16} is a 
compact manifold with a positive line bundle. 
Marcinkiewicz-Zygmund families for the space of holomorphic sections in 
powers of the 
line bundle are connected to sampling sequences  in the tangent 
space.

Though line bundles appear much more complicated objects than
Fock space, which even has a closed-form reproducing kernel, Fock
space presents some new  difficulties. It lacks  compactness 
that made off-diagonal estimates for the reproducing kernel easier in 
\cite{LOC16}. Another source of difficulty is the  behavior under 
translation. The Fock space $\cF^2$ is invariant with respect to
Bargmann-Fock shifts, while $\pln$ endowed with the 
Fock norm is not.

Finally we mention the extensive work on the asymptotics of
reproducing kernels for weighted polynomials and related construction
of \emph{random} \mz\ families in the context of determinental point
processes~\cite{Ameur13} -- \cite{AC21}.  
In \cite{AOC12,AR20} Y.\ Ameur and his coauthors  have studied a similar 
notion of sampling polynomials with respect to the discrete norm $\sum _{\lambda \in \Lambda _n} 
|p(\lambda)|^2 e^{-\pi n|\lambda|^2}$ instead of  $\sum _{\lambda \in \Lambda _n} 
\frac{|p(\lambda     )|^2}{k_n(\lambda,\lambda )}$. 
Note that  all this work uses   measures that  depend on the polynomial
degree $n$, quite in contrast to our set-up \eqref{eq:k1}.   Their  choice  was motivated  by problems arising in random 
Gaussian matrix ensembles and models of the distribution of points in the one 
component plasma. The common ground   is the
construction of point sets that are sampling for polynomials in Fock
spaces. In \cite{AOC12,AR20} this is achieved with random processes
(determinental point processes), whereas our construction is purely
deterministic and investigates the connection to the
infinite-dimensional sampling problem in $\cF ^2 $.  Therefore the results in
\cite{AOC12,AR20} are not directly comparable to ours.


One of our basic tools 
is the size of the reproducing kernel for the polynomials in a
weighted $L^2$-space. Since  our weight is the 
Gaussian weight, the kernel can be expressed explicitly in terms of the incomplete 
Gamma function which is a classical and well-studied object.  We have  
collected the necessary results  in the appendix for  the sake of
being self-contained. Estimates for the reproducing kernel have  been 
studied in great generality in \cite{Ameur13, AC21} with potential
theoretic methods --- without any reference to the incomplete Gamma
function. Possibly these estimates could also be used  in our context.

 The estimates for the intrinsic reproducing kernel $k_n$ show that
 (i)  the
$L^2$-energy of a  polynomial of degree 
$n$ is  concentrated in a disk of radius  
$\sqrt{n/\pi}$, the so-called bulk region, and (ii) that the intrinsic
kernel $k_n$ for $\pln $ is comparable to the kernel $k(z,w) = e^{\pi z
  \bar{w}}$ precisely in the bulk region. See Lemma 
Lemma~\ref{energy} and Corollary~\ref{kernequ} for  the  precise
statements.

As a consequence of Theorem~\ref{tmintro1} we mention an application
to \tfa . It is well-known that all problems about sampling  in Fock space possess an equivalent formulation about
Gabor frames in $L^2(\bR )$. To state this version, we denote the \tfs\
of a function $g$ by $z= (x,\xi )\in \bR ^2$ by $(\pi (z) g)(t) = e^{2\pi i
  \xi t} g(t-x)$ for $t,x, \xi \in \bR $. The $L^2$-normalized Hermite functions are
denoted by $h_n$, in particular $h_0(t) = 2^{1/4} e^{-\pi t^2}$ is the
Gaussian. Then Theorem~\ref{tmintro1}(i) is equivalent to the
following statement, which may be of interest in the time-frequency
analysis of signal subspaces~\cite{Hla98}. 

\begin{tm}
      Assume that $\Lambda $ is a sampling set for $\cF ^2$ and $\tau
      >0$ large enough.  Then $\{\pi
      (\lambda )h_0 : \pi |\lambda | ^2 \leq n+\sqrt{n}\tau \}$ is a
      frame for $V_n = \mathrm{span}\, \{h_k : k=0, \dots , n\}$ with
      bounds independent of $n$. This means that 
      $$
   A\|f\|_2 \leq \sum _{\lambda \in \Lambda : \pi |\lambda |^2 \leq
     n+\sqrt{n}\tau } |\langle f, \pi _\lambda h_0 \rangle |^2 \leq
   B\|f\|_2^2 \qquad \text{ for all } f\in V_n \, .
   $$
    \end{tm}

\textbf{Outlook.} It is needless to say that the topic of \mz\
families and sampling theorems admits dozens of variations. The
ultimate goal is to understand \mz\ families for polynomials $\pln $
in a weighted Bergman space on some general domain $X\subseteq \rd $ (or
$\subseteq \cd $). Intermediate problems would be \mz\ families for polynomials in Fock spaces with more general
   weight $e^{-Q(z)}$, or  the construction  \mz\ families for
   multivariate Bergman spaces $A^2(\mathbb{B} _n)$ in  $n$ complex variables on unit ball in $\bC ^n$.
Even simple variations of the set-up yield interesting new questions.

    The paper is organized as follows: In Section~2 we recall the basic
facts about the Fock space and the associated reproducing
kernels. Section~3 summarizes the required asymptotics of the
incomplete Gamma function. In Section~4 we relate sampling sets for
Fock space to \mz\ families and prove the first part of
Theorem~\ref{tmintro1}. Section~5 covers the converse statement. In
Section~6 we deal with uniform interpolating families and prove
Theorem~\ref{tmintro2}. The connection to the time-frequency analysis
of signal subspaces is explained in Section~7. Finally, in the
appendix we offer some elementary estimates for the zero-order
asymptotics of the incomplete Gamma function. These are, of course,
well-known and added only to make the paper self-contained.

\section{Fock  space}


The monomials $z\mapsto z^k$ are orthogonal in $\cF ^2$, and  the
normalized monomials 
$$
 e_k(z) = \Big( \frac{\pi ^k}{k!}\Big)^{1/2}  \, z^k
 $$
 form an orthonormal basis for $\cF^2$.

 Let $\cP _n$ be the subspace of polynomials of degree at most $n$ in
 $\cF ^2$. The reproducing kernel of $\pln $ is given by
 \begin{equation}
   \label{eq:rep}
   k_n(z,w) = \sum _{k=0}^n e_k(z) \overline{e_k(w)} = \sum _{k=0}^n
   \frac{(\pi z \bar{w})^k}{k!} \, .
 \end{equation}
As $n\to \infty $, this kernel converges to the reproducing kernel of
$\cF ^2 $:
$$
k(z,w) = \lim _{n\to \infty } k_n(z,w) = e^{\pi z \bar{w}} \, .
$$
As we have learned in our study of \mz\ families in Bergman spaces~\cite{GOC21}, we
will need to understand the relation of the kernel $k_n$ to $k$. For
this purpose we will make use of the properties and the asymptotics of
the \emph{incomplete Gamma function}
\begin{equation}
  \label{eq:2}
  \Gamma (z,a) = \int _{a} ^\infty t^{z-1} e^{-t} \, dt \, 
\end{equation}
and
\begin{equation}
  \label{eq:3}
  \gamma (z,a) = \int _0^{a} t^{z-1} e^{-t} \, dt \, .
\end{equation}
Denote the centered disc of radius $\rho $ by  $B_\rho = \{ z\in \bC  : |z| \leq
\rho \}$. Then
\begin{align}
  \int _{B_\rho } |z|^{2k} e^{-\pi |z|^2} \, dm(z) & = 2\pi \int _0^\rho
r^{2k} e^{-\pi r^2}  \, rdr \notag \\
&= \int _0^{\pi \rho ^2} \Big( \frac{u}{\pi}\Big)^k \, e^{-u} \, du
                                        \notag \\
&= \frac{1}{\pi ^k} \gamma (k+1,\pi \rho ^2) \, . \label{gamma1}
\end{align}

\begin{lemma} \label{rkf}
  We have
  $$
  k_n(z,w) = e^{\pi z \bar{w} } \frac{\Gamma (n+1,\pi z\bar{w})}{n!}
  $$
  In particular $k_n(z,z) = e^{\pi |z|^2} \frac{\Gamma(n+1,\pi
    |z|^2)}{n!}$.~\footnote{Note that we consider $\pln $ as a
    subspace of $\cF ^2$ and always use the fixed weight  $e^{-\pi |z|^2}$. The work
  on determinental point processes always uses the weight $e^{-\pi n
    |z|^2}$ for $\pln $. The formulas and the asymptotics are
  therefore different.}  
\end{lemma}
\begin{proof}
  See   \cite[8.4.8]{NIST}, or use the obvious formula
  $$
  \frac{1}{n!} \Gamma (n+1, r) = \frac{1}{n!} \int _r^\infty t^n
  e^{-t}\, dt = \frac{r^n}{n!} e^{-r} +   \frac{1}{(n-1)!} \Gamma (n,
  r)
  $$
  repeatedly and then use analytic extension and  substitute $r= \pi
  z\bar{w}$. 
\end{proof}

The energy of a polynomial  $p(z) = \sum _{k=0}^n a_k z^k \in \pln $
on a disc $B_\rho $ is
\begin{align*}
  \int _{B_\rho } |p(z)|^2 \, e^{-\pi |z|^2} \, dm(z) &= \sum _{k=0}^n
 |a_k|^2 \int _{B_\rho } |z|^{2k} e^{-\pi |z|^2}\, dm(z) \\ 
&= \sum _{k=0}^n |a_k|^2 \frac{k!}{\pi ^k} \, \frac{\gamma (k+1,\pi
  \rho ^2)}{k!} \\
&\geq \min _{0\leq k\leq n}  \frac{\gamma (k+1,\pi   \rho ^2)}{k!} \, \sum _{k=0}^n |a_k|^2 \frac{k!}{\pi ^k} 
&\geq \frac{\gamma (n+1,\pi   \rho ^2)}{n!} \|p\|_{\fc } ^2 \, .
\end{align*}
In the last inequality we have used the fact that $k\to \frac{\gamma
  (k+1,\pi   \rho ^2)}{k!}$ is decreasing and that $\|z^k\|_{\fc }^2 =
k!/\pi ^k$ by \eqref{gamma1}.

Consequently, the energy of a polynomial $p\in \pln $ is concentrated
$B_\rho $ with $\rho = \sqrt{n/\pi}$, the so-called bulk region.
\begin{lemma}
  \label{energy}
  For every $p \in \pln $ we have
  \begin{equation}
     \label{eq:w1}
  \int _{B_\rho ^c} |p(z)|^2 \, e^{-\pi |z|^2} \, dm(z) \leq   \frac{\Gamma (n+1,\pi \rho ^2)}{n!} \|p\|_{\fc } ^2 \, . 
   \end{equation}
\end{lemma}
\begin{proof}
  This follows immediately from the previous estimates via 
  \begin{align}
  \int _{B_\rho ^c} |p(z)|^2 \, e^{-\pi |z|^2} \, dm(z)  &= \|p\|^2_{\fc }
  - \int _{B_\rho } |p(z)|^2 \, e^{-\pi |z|^2} \, dm(z) \notag \\
 &\leq (1-
  \frac{\gamma (n+1,\pi   \rho ^2)}{n!} ) \, \|p\|_{\fc } ^2 =
  \frac{\Gamma (n+1,\pi \rho ^2)}{n!} \|p\|_{\fc } ^2 \, . 
\end{align}
\end{proof}

\section{Asymptotics of the incomplete Gamma function}
The asymptotic behavior of the incomplete Gamma function is well
understood. We collect the properties required for \mz\ families in
Fock space. As usual $f \asymp g $ means that there exists a constant
such that $C\inv f (x) \leq g(x) \leq C f(x)$ for all $x$ in the
domain of $f$ and $g$, $f \lesssim g$ means $f(x) \leq C g(x)$,  and $f \sim g$ near $x_\infty $  means that
$\lim _{x\to x_\infty} \frac{f(x)}{g(x)} = 1$.

The following result has been proved on several levels of generality~\cite{GST12,Tem75,Tri50,Nem16,Nem16a}.

 \emph{ The normalized incomplete Gamma function admits the asymptotic
  expansion
  \begin{equation}
    \label{eq:w11}
 \frac{\Gamma(a, a + \tau \sqrt{a})}{\Gamma (a)} \sim \tfrac{1}{2} 
\mathrm{erfc} \, (\frac{\tau}{\sqrt{2}}) + \frac{1}{\sqrt{2\pi a}}
 e^{-\tau ^2/2} \sum _{n=0}^\infty \frac{C_n(\tau )}{a^{n/2}} \, , 
  \end{equation}
where $C_0(\tau ) = \frac{\tau^2 - 1}{3}$ and $\mathrm{erfc}(y) =
\frac{2}{\sqrt{\pi }}\int
_y^\infty e^{-t^2} \, dt $. 
This implies that} 
\begin{equation}
  \label{eq:w12}
\Big|\frac{\Gamma(a,a + \tau \sqrt{a})}{\Gamma (a)} - \tfrac{1}{2}
\mathrm{erfc} (\frac{\tau}{\sqrt{2}})\Big| \leq \frac{\tau^2 - 1}{3}
    e^{-\tau   ^2/2} \frac{1}{\sqrt{2\pi a}} \, .
\end{equation}
We only need these estimates for $a=n+1$ and $\tau >0$, but their validity has
been established for  large domains in $\bC $. 

\begin{prop} \label{asympt1}~
\begin{enumerate}[(i)]
\item  For every $\epsilon >0$ there is $\tau>0$, such that
 $$
\frac{ \Gamma (n+1,n+\sqrt{ n }\tau)}{n!} < \epsilon \qquad \forall
n\geq n_0 $$
\item  For every $\tau >0$ there is a constant $ C(\tau )>0$, such that  
$$
\frac{ \Gamma (n+1,n+\sqrt{ n }\tau )}{n!} \geq C(\tau ) \qquad \forall
n\geq n_0 \, .$$
In fact, $C(\tau )$ can be taken as $C(\tau ) =  \frac{1}{4}
\mathrm{erf} (\tau /\sqrt{2})  $
\item  For $\tau >0$
$$
1-\frac{1}{n!} \Gamma (n+1, n-\sqrt{n}\tau ) \leq
 e^{-\tau ^2/2} \, .
$$
\item  For every  $x\geq 0$
\begin{equation*}
  \lim _{n\to \infty } \frac{\Gamma (n+1,x)}{n!} =1  \, .
\end{equation*}
The convergence is uniform on bounded sets $\subseteq \bR ^+$ and
exponentially fast. 

\item For all $n$
\begin{equation*}
 \frac{\Gamma (n+1,n)}{n!} > 1/2 \,. 
\end{equation*}
\end{enumerate}
 \end{prop}

 \begin{proof}
 Items (i) and (ii) follow readily  from \eqref{eq:w12} as follows:  

 (i)    Choose $\tau >0$, such that
   $\tfrac{1}{2}\mathrm{erfc}(\tau/\sqrt{2})  < \epsilon /2$. Now
   choose $n_0\in \bN $, such that the error $\frac{\tau^2 - 1}{3}
    e^{-\tau   ^2/2} \frac{1}{\sqrt{2\pi (n+1)}} < \epsilon /2$ for
    $n\geq n_0$. By \eqref{eq:w12} we then have $\frac{ \Gamma
      (n+1,n+\sqrt{ n }\tau)}{n!} < \epsilon $ for all $n\geq n_0 $. 

(ii) Given $\tau >0$ choose $n_0$ such that the error  $\frac{\tau^2 - 1}{3}
    e^{-\tau   ^2/2} \frac{1}{\sqrt{2\pi (n+1)}} < \tfrac{1}{4}
    \textrm{erfc} (\tau /\sqrt{2}) $ for $n\geq n_0$. Then 
$\frac{ \Gamma (n+1,n+\sqrt{ n }\tau )}{n!} \geq  \tfrac{1}{4}
    \textrm{erfc} (\tau /\sqrt{2})>0$  for all
$n\geq n_0 $. 

(iii) and (iv) are well-known.

(v) is taken from ~\cite{NIST}, formula 8.10.13.

For completeness we summarize the   arguments for the zero
order asymptotics in the appendix. In contrast to the  full
asymptotics of the incomplete Gamma function, they are elementary.  
 \end{proof}

 \begin{cor} \label{kernequ}
If  $\pi |z|^2 \leq n + \sqrt{n}\tau$, then  $k_n(z,z) \asymp k(z,z)
= e^{\pi |z|^2}$ with constant depending only on $\tau $, but not on
$n$.
\end{cor}

 We also need an off-diagonal estimate for the kernel $k_n$.
 \begin{lemma} \label{offd}
   Assume that $\pi |z|^2 < n(1-\epsilon)  $  for fixed $\epsilon >0$
   and $|z-w | \leq \tau $. Then for $n$ large enough depending on
   $\epsilon $, 
   \begin{equation}
     \label{eq:de1}
     \Big|\frac{\Gamma(n+1, \pi z\bar{w}) }{n!} -
     \frac{\Gamma(n+1, \pi |z|^2 )}{n!} \Big| \leq  C e^{-\epsilon ^2 n /4 } \, .
   \end{equation}
\end{lemma}
 \begin{proof}
Since $\Gamma (n+1, \pi z\bar{w})$  is invariant with respect to the
rotation $(z,w) \to (e^{i\theta } z,e^{i\theta }w)$, we may assume
that $z=r \in \bR, r>0$ and $w=r+\bar{u}$ with $|u| \leq \tau$.  
   We then write
   $$
   \Gamma (n+1, \pi z\bar{w}) =    \Gamma (n+1, \pi r^2 + \pi r u) =
   \int _{\pi r^2} ^\infty t^n e^{-t} \, dt +    \int _{\pi r^2 + \pi
     ru} ^{\pi r^2}  \dots = \Gamma (n+1, \pi r^2) +  \int _{\pi r^2 + \pi
     ru} ^{\pi r^2} \dots \, .
$$
Let $\gamma (s) = \pi r^2 + s \pi r u$ the line segment from $\pi r^2
\in \bC $ to $\pi r^2 + \pi r u$. Then  
\begin{align*}
\big|\int _{\pi r^2}  ^{\pi r^2 + \pi
     ru} t^n e^{-t} \, dt \big|&= \big|\pi r u \int _0^1 (\pi r^2 + s \pi ru)^n
     e^{-\pi r^2 - s \pi ru} \, ds\big| \\
   &\leq \pi r |u| (\pi r^2 + \pi r|u|)^n e^{-\pi r^2+\pi r |u|} \\
&\leq \pi r \tau (\pi r^2 + \pi r \tau)^n e^{-\pi r^2 + \pi r \tau }   \, .
  \end{align*}
Observe that $r\to (\pi r^2 + \pi r \tau)^n e^{-\pi r^2}$ is
increasing, as long as $\pi r^2 + \pi r \tau \leq n$.  Set $x=\tfrac{\pi r^2}{n}\leq 1-\epsilon$, then $\pi r \tau =\sqrt{\pi n x } \tau$,  and  by assumption $x<1$. Using Sterling's formula, we continue with
 \begin{align*}
\frac{1}{n!} \,  \pi r \tau (\pi r^2 + \pi r \tau)^n e^{-\pi r^2 +
   \pi r \tau  }  &\leq \frac{\sqrt{\pi } \tau
   \sqrt{nx}}{\sqrt{2\pi n }} \, \Big(\frac{e}{n} \Big)^n \Big(nx +
   \sqrt{\pi nx} \tau\Big)^n e^{-nx + \sqrt{\pi x n} \tau }  \\
&\lesssim \sqrt{x} \Big(x+\frac{\sqrt{\pi x} \tau}{\sqrt{n}}\Big)^n  e^{n-nx
         +   \sqrt{\pi x n} \tau}\\
&\leq x^n \Big(1+\frac{\sqrt{\pi } \tau}{\sqrt{xn}}\Big)^n
                                      e^{n(1-x+\sqrt{\pi x}\tau / \sqrt{n})} \\
&\leq \exp \Big(n\big(   1 - x+   \ln x +  \ln (1 + \frac{\sqrt{\pi}
 \tau}{\sqrt{n  x}}) + \frac{\sqrt{\pi}  \tau}{\sqrt{n  x}} \big)\Big)  \, .
 \end{align*}
Since $x\to 1-x+\ln x$ is increasing on $(0,1]$ and $x\leq 
1-\epsilon $, we have $1-x+\ln x \leq \epsilon +\ln (1-\epsilon )
\leq - \epsilon ^2/2$. Choose $n$ so large that $\ln (1 + \frac{\sqrt{\pi}
 \tau}{\sqrt{n  x}}) + \frac{\sqrt{\pi}
 \tau}{\sqrt{n  x}} \leq \epsilon ^2/4 $, then  the latter expression is
dominated by
$e^{ -n \epsilon^2/4}$,  
and   this expression tends to $0$ exponentially fast, as $n\to
\infty $. This proves the claim.
 \end{proof}

\section{Sampling implies \mz\ inequalities}

We summarize the main facts about  sampling sets in $\cF ^2$ from the
literature~\cite{lyub92,seip92,seip04,seip-wallsten}.

(i) A set  $\Lambda \subseteq \bC $ is sampling for
$\cF ^2$, \fif\  it contains a uniformly separated set $\Lambda
'\subseteq \Lambda $ with lower Beurling density  $D^-(\Lambda ' ) >1$. 


(ii) \emph{Tail estimates.} Let  $f\in \cF ^2$ and $\rho >0$. The subharmonicity of $|f|^2$
implies that
\begin{equation}
  \label{eq:r4}
  |f(\lambda )|^2 e^{-\pi |\lambda |^2} \leq c_\rho \int _{B(\lambda ,
    \rho )} |f(z)|^2 \, e^{-\pi |z|^2} \, dm(z)
\end{equation}
for all $\lambda \in \bC $. The constant is $c_\rho = e^{\pi \rho
  ^2}/(\pi \rho ^2)$, but we will not need it.

(ii) If $\Lambda $ is relatively separated, i.e., a finite union of
$K$ uniformly discrete subsets of $\bC $ with separation $\rho >0$, then
\begin{equation}
  \label{eq:r6}
\sum _{\lambda \in \Lambda , |\lambda | >R}    |f(\lambda )|^2 e^{-\pi
  |\lambda |^2} \leq c_\rho K \int _{|z|> R-\rho}  |f(z)|^2 \, e^{-\pi |z|^2} \, 
dm(z)
\end{equation}

\begin{tm} \label{samptomz}
  Assume that $\Lambda \subseteq \bC $ is a sampling set for $\cF ^2
  $ with bounds $A,B$.  For $\tau >0$ set $\rho _n$, such that $\pi \rho _n^2 =
  n+\sqrt{n} \tau $. Then for  $\tau >0$ large enough,   the sets $\Lambda _n = 
  \Lambda \cap  B_{\rho_n} $  form a \mz\ family for $\pln $ in
  $\cF ^2$.
\end{tm}

\begin{proof}
\emph{Lower bound}: Since always $k_n(z,z) \leq k(z,z) = e^{\pi
  |z|^2}$, we may replace $k_n$ by $k$ in the sampling inequalities:
\begin{align*}
\sum _{\lambda \in \Lambda _n} \frac{|p(\lambda
  )|^2}{k_n(\lambda,\lambda)} &\geq \sum _{\lambda \in \Lambda _n} \frac{|p(\lambda
      )|^2}{k(\lambda,\lambda)}\\
  &= \sum _{\lambda \in \Lambda _n} |p(\lambda
)|^2 e^{-\pi |\lambda |^2} = \sum _{\lambda \in \Lambda } - \sum
    _{\lambda \in \Lambda : |\lambda | > \rho _n} \dots  \, .
  \end{align*}
Since $\Lambda $ is a sampling set for $\cF ^2$, the first term 
satisfies
$ \sum _{\lambda \in \Lambda } |p(\lambda
)|^2 e^{-\pi |\lambda |^2} \geq A \|p\|\fc ^2$. For the second term we
observe that $\Lambda $ is a finite union of uniformly discrete sets
with separation $\rho >0$ and apply  \eqref{eq:r6} and \eqref{eq:w1}:
\begin{align*}
  \sum _{\lambda \in \Lambda : |\lambda | >\rho _n} |p(\lambda
)|^2 e^{-\pi |\lambda |^2} &\leq C \int _{|z|> \rho _n-\rho} |p(z)|^2 
e^{-\pi |z|^2}\, dm(z) \\
&\leq C \frac{\Gamma (n+1, \pi (\rho _n-\rho )^2)}{n!}  \|p\|\fc ^2
                         \qquad \text { for } p\in \pln \, .
\end{align*}
Our choice of $\rho _n$ implies that
\begin{align*}
\pi (\rho _n-\rho )^2 &= \pi \rho _n^2 - 2\pi \rho _n \rho + \pi \rho
                        ^2 \\
&=  n+\sqrt{n} \tau  + \pi \rho ^2 - 2\sqrt{\pi } \,  \sqrt{n+\sqrt{n}\tau} \rho \\
&\geq   n+\sqrt{n} \tau  - 2\sqrt{\pi n} \rho (1+\frac{\tau}{2\sqrt{n}})\\
&\geq  n+\sqrt{n} \tau '
\end{align*}
with $\tau' = \tau - 3\sqrt{\pi } \rho  $ whenever 
$\sqrt{n} \geq \tau$. 
 Since $a\mapsto \Gamma
(x,a)$ is decreasing, we have
$\frac{\Gamma (n+1, \pi (\rho _n-\rho )^2)}{n!} \leq \frac{\Gamma (n+1,
  n+\sqrt{n}\tau ')}{n!}$. In view of Corollary~\ref{asympt1}(i) we may  choose
$\tau ' $ and hence $\tau $ so that 
$$
  \sum _{\lambda \in \Lambda : |\lambda | >\rho _n} |p(\lambda
)|^2 e^{-\pi |\lambda |^2} \leq C \frac{\Gamma (n+1,
  n+\sqrt{n}\tau ')}{n!}  \|p\| \fc ^2 \leq \frac{A}{2} \|p\|\fc  \qquad \text{
  for all } p \in \pln  \, 
$$
for large $n$, $n\geq n_0$, say. Combining the inequalities, we obtain $
\sum _{\lambda \in \Lambda _n} \frac{|p(\lambda
  )|^2}{k_n(\lambda,\lambda)} \geq \frac{A}{2}\|p\|\fc ^2$ for all
$p\in \pln $.

\emph{Upper inequality:} For the above  choice of $\tau $
Proposition~\ref{asympt1}(ii)   says that 
$$
\frac{\Gamma (n+1,   n+\sqrt{n}\tau )}{n!} \geq   
\tfrac{1}{4}     \textrm{erfc} (\tau /\sqrt{2}) =C(\tau ) = C \, 
$$
for $n\geq n_0$. 
This implies that $k_n(\lambda , \lambda )\inv  \leq C\inv k(\lambda ,
\lambda )\inv = C \inv e^{-\pi |\lambda |^2} $ for $\pi |\lambda |^2
\leq n+\sqrt{n} \tau $,  and thus 
\begin{equation*}
        \sum _{\lambda \in \Lambda _n} \frac{|p(\lambda
  )|^2}{k_n(\lambda , \lambda )} \leq  C\inv     \sum _{\lambda \in
  \Lambda  , \pi |\lambda |^2 \leq  n+\sqrt{n}\tau} |p(\lambda )|^2 e^{-\pi
  |\lambda |^2}  \leq C\inv B  \|p\| \fc \, \quad \text{ for all } p \in
\pln \, ,
\end{equation*}
because $\Lambda \supseteq \Lambda _n$ is a sampling set for $\cF
^2$. 
\end{proof}

Note that the lower bound in the \mz\ inequalities matches the lower
bound $A$ of the sampling inequality in $\cF ^2$, whereas the upper
bound is $4 \,\mathrm{erfc}\, (\tau /\sqrt{2})\inv B$ depends also on
the additional parameter $\tau $. 


\begin{cor} \label{mzsampcount}
  For every $\epsilon >0$ there exist \mz\ families $(\Lambda _n)$ for
  $\pln $ in $\cF ^2$ with $\# \Lambda _n \leq (1+\epsilon ) (n+1)$
  points. 
\end{cor}
\begin{proof}
Choose $\mu , \delta   $ small enough, so that $(1+2\mu ) (1+\delta ) <
1+\epsilon $.   Let $\Lambda \subseteq \bC  $ be a (uniformly) discrete subset  with
  $D^-(\Lambda )>1$ and $D^+(\Lambda ) < 1+\mu $. Then $\Lambda $ is a
  sampling set for $\cF $ by the characterizations of
  Lyubarskii~\cite{lyub92} and Seip~\cite{seip92,seip-wallsten}, and for $\pi \rho
  _n 
  ^2 = n+ \sqrt{n}\tau $  the sets $\Lambda \cap B_{\rho _n}$ form a
  \mz\ family for $\pln $. For  $n$ large enough and  $\tau/\sqrt{n}
  <\delta $, 
we  find that 
  $$
  \# \Lambda _n = \# (\Lambda \cap B_{\rho _n}) \leq (1+2\mu )
  |B_{\rho _n}| = (1+2\mu ) (n+ \sqrt{n}\tau )  < (1+\epsilon )
  (n+1)\, .
  $$
\end{proof}

For a \mz\ family for $\pln$ we need at least $\mathrm{dim} \, \pln =
n+1$ points in each layer $\Lambda _n$. The construction above yields
\mz\ families for Fock space with  nearly optimal cardinality.

\vspace{2mm}


\section{ \mz\ inequalities imply sampling}

We first formulate a few properties of the distribution of \mz\ families. 

\begin{lemma} \label{distr}
Assume that $(\Lambda _n)$ is a \mz\ family for $\cF ^2$ with bounds
$A,B$. Let $\varepsilon >0$ and   $\pi \sigma _n^2 = n(1-\varepsilon)$. 

(i)  Then 
 $\# (\Lambda _n \cap B_{\sigma _n}^c) \leq B
 \frac{n+1}{(1-\varepsilon)^n}$. This holds also for $\epsilon =0$. 

(ii)  
Let   $ B(z,\rho )$ be 
a disc in $B_{\sigma _n}$. Then
$$
\# (\Lambda _n \cap B(z,\rho  )) \leq C \, .
$$

Consequently, every  $\Lambda _n\cap B_{\sigma 
    _n}$  
  is a union of at most $L$
  separated sets with uniform separation $\delta >0$ independent of
  $n$.

\end{lemma}

  \begin{proof}
  (i)   For $\pi |w|^2
\geq n$ and $k\leq n$,  we have 
$$
\frac{(\pi |w|^2)^k}{k!} \leq \frac{(\pi |w|^2)^n}{n!} \, ,
$$
so the reproducing kernel satisfies the estimate 
\begin{equation}
  \label{eq:w13}
k_n(w,w) = \sum _{k=0}^n \frac{(\pi |w|^2)^k}{k!} \leq (n+1)
\frac{(\pi |w|^2)^n}{n!} \, . 
\end{equation}

If $\pi|z|^2\ge (1-\varepsilon) n$, then, using
$\sqrt{1-\varepsilon}\, w = z$, we have
$\pi |w|^2 \geq n$,  and as before we obtain:
\[
 k_n(z,z) \le \frac{n+1}{(1-\varepsilon)^n} \frac{(\pi |z|^2)^n}{n!},
\]

To estimate $\# (\Lambda _n \cap B_{\sigma _n}^c)$, we test the \mz\
inequalities for the monomial $p_n(z) = \frac{\pi 
  ^{n/2}}{n!^{1/2}} z^n$. Then $\|p_n\|\fc = 1$. 
  \eqref{eq:w13} implies that 
$$
\frac{|p_n(\lambda )|^2}{k_n(\lambda ,\lambda )} \geq 
\frac{(1-\varepsilon)^n}{n+1}
\qquad \pi |\lambda |^2 \geq n(1-\varepsilon) \, ,
$$
and therefore 
 \begin{align*}
 \frac{(1-\varepsilon)^n}{n+1} \# (\Lambda _n \cap B_{\sigma _n}^c) \leq    
\sum _{\lambda \in \Lambda _n : \pi |\lambda |^2 \geq n(1-\varepsilon)}
   \frac{|p_n(\lambda )|^2}{k_n(\lambda ,\lambda )} \leq B \|p_n\|\fc = B  
 \end{align*}

 we obtain $ \# (\Lambda _n \cap B_{\sigma_n}^c) 
\leq \frac{B(n+1)}{(1-\varepsilon)^n} $.


(ii) 
Let $\kappa _{n,z}(w) = k_n(w,z)/k_n(z,z)^{1/2}$ be the
normalized reproducing kernel of $\pln $ and $ B(z,\rho )\subseteq
B_{\sigma _n}$ be an arbitrary  disc inside $B_{\sigma _n}$. Recall that $k_n(z,w) = e^{\pi \bar{z}w}
\Gamma (n+1,\pi \bar{z}w)/n!$ and that for $\pi |z|^2 \leq
n(1-\epsilon )$ we have
$\Gamma (n+1,\pi |z|^2)/n! \geq 1/2$ by Proposition~\ref{asympt1}(v). 
So after substituting the formulas for the kernel, we obtain
\begin{align*}
  \sum _{\lambda \in \Lambda _n \cap B(z,\rho)} \frac{|\kappa _{n,z}(\lambda
  )|^2}{k_n(\lambda,\lambda)}&=   \sum _{\lambda \in \Lambda _n \cap B(z,\rho)}
 \frac{|k_n(z,\lambda)|^2}{k_n(z,z) k_n(\lambda,\lambda )}\\
&=   \sum _{\lambda \in \Lambda _n \cap B(z,\rho)} |e^{\pi \bar{z}\lambda }|^2
e^{-\pi |z|^2}e^{-\pi |\lambda|^2} \frac{|\Gamma (n+1,\pi  \bar{z}\lambda)|^2}{\Gamma 
(n+1,\pi |z|^2) \Gamma (n+1,\pi |\lambda |^2)}\\
&=  \sum _{\lambda \in \Lambda _n \cap B(z,\rho )} e^{-\pi |\lambda -z|^2}
\frac{|\Gamma (n+1,\pi  \bar{z}\lambda)|^2}{\Gamma 
(n+1,\pi |z|^2) \Gamma (n+1,\pi |\lambda |^2)} = (*)\, .
\end{align*}
Now note that by Lemma~\ref{offd} $|\Gamma (n+1,\pi
\bar{z}\lambda)|^2/n! \geq 1/4$ for $n$ large, whereas $|\Gamma
(n+1,\pi  |z|^2)/n! \leq 1$, so that the last sum finally is bounded
below by
$$
(*) \geq e^{-\pi \rho ^2} \sum _{\lambda \in \Lambda _n \cap B(z,\rho) }
\frac{1}{4} = \frac{1}{4} \, e^{-\pi \rho ^2} \, \# (\Lambda _n \cap B
)\, . 
$$
Reading backwards,  we obtain
$$
\frac{1}{4} \, e^{-\pi \rho ^2} \, \# (\Lambda _n \cap B ) \leq  \sum _{\lambda \in 
\Lambda _n \cap B} \frac{|\kappa _{n,z}(\lambda
)|^2}{k_n(\lambda,\lambda)} \leq B \|\kappa _{n,z}\|_{\cF }  ^2 = B \,
, 
$$
which was claimed. 
\end{proof}




For completeness we mention  that the number of  points in the
transition region  $C_{n,\tau } = \{ z\in \bC : n-\sqrt{n}\tau \leq
 \pi |z|^2 \leq n+\sqrt{n} \tau \}$ is bounded by 
$$
\# (\Lambda _n \cap C_{n,\tau} ) \lesssim  \sqrt{n} e^{\tau ^2/2} \, .
$$
This can be shown as above by testing against the monomial $z^n$.

Before stating our main theorem, we 
recall that a sequence of sets $\Lambda _n\subseteq \bC $ converges weakly to  
$\Lambda \subseteq \bC  $, if  for all compact disks  $B\subseteq \bC $
$$
\lim _{n\to \infty } d\big((\Lambda _n \cap B) \cup \partial B, (\Lambda  \cap B) \cup
\partial B \big) = 0 \, , 
$$
where $d(\cdot , \cdot )$ denotes the Hausdorff distance between two
compact sets in $\bC$.  If every  $\Lambda _n$ is the union of at most $K$
uniformly separated sets with fixed separation $\delta $,  then 
\begin{equation}
  \label{eq:fin1}
\sum _{\lambda \in \Lambda _n \cap B}
\frac{|f(\lambda)|^2}{k(\lambda,\lambda )} \to \sum _{\lambda \in \Lambda  \cap B}
\frac{|f(\lambda)|^2}{k(\lambda,\lambda )} m(\lambda )\, ,
\end{equation}
with multiplicities $\mu (\lambda ) \in \{1, \dots , K\}$. 

\begin{tm} \label{mzsampa}
  Assume that
  $(\Lambda _n)$ is a \mz\ family for the polynomials $\pln $ in
  $\cF ^2$. Let $\Lambda $ be a weak limit of $(\Lambda
  _n)$ or of some subsequence $(\Lambda _{n_k})$.
  Then $\Lambda $ is a sampling set for $\cF ^2 $. 
\end{tm}
\begin{proof}
The assumption that  $\Lambda_n$ is a Marcinkiewicz Zygmund family for
$\pln$ in $\cF ^2$  means that  there exist   $A,B> 0$ such that 
$A\|p\| \fc \leq \sum _{\lambda \in \Lambda _n} \frac{|p(\lambda
    )|^2}{k_n(\lambda,\lambda)}\leq B\|p\|^2\fc  $ for all polynomials
  $p\in \cP _n$.
  
(i) Let $B= \bar{B}(w,\rho )$ be a closed  disc. 
 By Lemma~\ref{distr}(ii) $\# (\Lambda _n \cap B(w,\rho )) \leq C$ for
some constant $C$ independent of $n$ and $B$,  provided that  $n$ is big enough. Since 
$\Lambda $ is a
weak limit of $\Lambda _n$, we know that $\# (\Lambda \cap \bar{B}(w,\rho)) \leq
C$. This means that $\Lambda $ is a union of $K$ uniformly separated
sets with separation $\delta >0$.

(ii) It follows immediately from \eqref{eq:r6} that $\Lambda $
satisfies the upper bound in the sampling inequality for $\cF ^2$. 

(iii) \textbf{Lower bound.} Fix a polynomial $p\in \cP _N$ (of degree $N$) and 
choose $r>0$
such that
$$
\int _{|z| \geq \sqrt{r/\pi } } |p(z)|^2 e^{-\pi |z|^2} \, dm(z) <
\frac{A}{4 c_\delta K} \|p\|_{\fc } ^2 \, ,
$$
where $c_\delta $ is the constant in ~\eqref{eq:r4} for separation
$\delta $. To avoid the ugly notation in subscript, we write $\nu =
\sqrt{r/\pi }$, $\rho _n = \sqrt{n/\pi }$, and $\sigma_n = 
\sqrt{n(1-\varepsilon)/\pi}$.   

For $p\in \cP _N$ the \mz\ inequalities are satisfied for every $n\geq
N$, therefore
\begin{align*}
  A\|p\|_{\fc } ^2 &\leq \sum _{\lambda \in \Lambda _n}  \frac{|p(\lambda
                     )|^2}{k_n(\lambda,\lambda)} \\
 &=  \sum _{\lambda \in \Lambda _n, |\lambda| <\nu +\delta } \frac{|p(\lambda
                 )|^2}{k_n(\lambda,\lambda)} \\
  &\qquad + \sum _{\lambda \in \Lambda _n, \nu +\delta \leq
      |\lambda| < \sigma_n } \frac{|p(\lambda
    )|^2}{k_n(\lambda,\lambda)} + \sum _{\lambda \in \Lambda _n,
      |\lambda | \geq \sigma_n   } \frac{|p(\lambda
    )|^2}{k_n(\lambda,\lambda)} = \\
   & = A_n + B_n + C_n \, .
\end{align*}
If $ |\lambda | \leq \sigma _n$, then $k_n(\lambda ,\lambda ) \geq
\frac{1}{2} k(\lambda,\lambda) = \frac{1}{2} e^{\pi |\lambda|^2}$ as a
consequence of Lemma~\ref{rkf} and Proposition~\ref{asympt1}(v) 
 Thus in the expressions for $A_n$ and $B_n$ we
may replace the kernel $k_n$ for polynomials by the kernel $k(z,z) =
e^{\pi |z|^2}$ for Fock space.
Consequently
\begin{align} \label{eq:fin2}
A \|p\|^2 \fc \leq 2  \sum _{\lambda \in \Lambda _n , |\lambda|
  \leq \nu +\delta } |p(\lambda )|^2 e^{-\pi |\lambda
  |^2} +B_n+C_n. 
\end{align}
Since in this sum   all points $\lambda$ lie in the compact set
$\bar{B}(0,\nu +\delta )$, 
the  weak convergence  (including
multiplicities $m(\lambda ) \in \{1, \dots, K\}$) implies 
the convergence  to $\Lambda$ and 
\begin{align*}
\lim _{n\to \infty } A_n &\leq 2  \lim_{n\to\infty}\sum_{\lambda \in \Lambda _n 
, |\lambda | \leq
  \nu +\delta } |p(\lambda )|^2 e^{- \pi|\lambda|^2} = \\
    & = 2     \sum_{\lambda \in \Lambda \cap \overline{B_{\nu +\delta }}} 
|p(\lambda
    )|^2 e^{- \pi|\lambda|^2} m(\lambda) 
\end{align*}
For the term  $B_n$, we recall  that every  $\Lambda_n \cap
B_{\sigma _n} $ is  a finite union of at most $K$ uniformly separated
sequences with separation $\delta $ and apply the tail estimate~\eqref{eq:r6}. Our
choice of $r$ and $\nu =
\sqrt{r/\pi }$   yields
\begin{align*}
B_n &\leq 2  \sum _{\lambda \in \Lambda _n, \nu+\delta \leq |\lambda| < \sigma_n 
} \frac{|p(\lambda  )|^2}{k(\lambda,\lambda)} \le 2 c_\delta K \int_{|z|>
    \nu  } |p(z)|^2 e^{-\pi |z|^2}\, dm(z) \\ & \leq 2 c_\delta K  \frac{A}{4
    c_\delta K} \|p\|\fc ^2  = \frac{A}{2} \|p\|\fc  ^2\,
  .
\end{align*}

To treat $C_n$, recall that $p$ has degree $N<n$. We use the trivial
estimate 
$$
|p(\lambda )|^2 = |\langle p, k_N (\lambda , \cdot )\rangle
|^2 \leq \|p\|_{\fc } ^2 k_N(\lambda , \lambda )
$$
and substitute into $C_n$ to obtain
$$
C_n = \sum _{\lambda \in \Lambda _n,  |\lambda| > \sigma _n} \frac{|p(\lambda
    )|^2}{k_n(\lambda,\lambda)} \leq \|p\|_{\fc } ^2 \, \# (\Lambda _n\cap 
B_{\sigma_n}^c)
   \sup _{|z| \geq \sigma _n} \frac{k_N(z,z)}{k_n(z,z)}
  \, .
$$
By Lemma~\ref{distr}(ii) $\# \Lambda _n\cap B_{\sigma_n}^c
    \leq \frac{B(n+1)}{(1-\varepsilon)^n}$, whereas the ratio of the
  different reproducing kernels is
$$
  \frac{k_N(z,z)}{k_n(z,z)} = \frac{e^{\pi |z|^2} \Gamma (N+1, \pi
    |z|^2) n!}{e^{\pi |z|^2} \Gamma (n+1, \pi |z|^2) N!} \, .
  $$
  For simplicity set $\pi |z|^2=R>n(1-\varepsilon)$. Then
  \begin{align*}
    \Gamma (n+1,R) &= \int _R ^\infty t^n e^{-t} \, dt \\
    &\geq R^{n-N} \int _R ^\infty t^N e^{-t} \, dt \\
&= R^{n-N} \Gamma(N+1,R) \, ,
  \end{align*}
  so that
  $$
  \sup _{\pi |z|^2 >n(1-\varepsilon)}    \frac{k_N(z,z)}{k_n(z,z)}  \leq 
(n(1-\varepsilon))^{N-n}\frac{n!}{N!} \, .
  $$
  Altogether
  $$
  C_n \leq \|p\|_{\fc } ^2 \frac{B(n+1)}{(1-\varepsilon)^n}
(n(1-\varepsilon))^{N-n}\frac{n!}{N!}\to 0 \, ,
  $$
  as $n\to \infty $ by Stirling's formula, provided that  we choose $\varepsilon$ such that 
$(1-\varepsilon)^2 > 1/e$.

Combining the estimates for $A_n,B_n$, and $C_n$ and letting $n$ go to
$\infty $, we obtain the lower sampling inequality
\begin{align*}
\sum _{\lambda \in \Lambda } |p(\lambda )|^2 m(\lambda ) e^{-\pi
  |\lambda |^2} &\geq  \sum _{\lambda \in \Lambda , |\lambda | \leq \nu +\delta 
} 
|p(\lambda )|^2 m(\lambda ) e^{-\pi
  |\lambda |^2} - \limsup _{n\to \infty } B_n - \lim _{n\to \infty } C_n
\geq\\ 
&^\geq \frac{A}{2}\|p\|\fc ^2 \, .
\end{align*}
As the multiplicities satisfy $1\leq m(\lambda ) \leq K$ for $\lambda \in
\Lambda $, we may omit them by changing the lower sampling constant to
$A/(2K)$. 

Since polynomials are dense in $\cF ^2$, this estimate extends to all
of $\cF ^2$. 
\end{proof}

\section{Uniform Interpolation}

In a sense the dual problem to sampling  is the  interpolation of
function values. A set $\Lambda \subseteq \bC  $
is interpolating for  $\cF ^2$, if for every $a=(a_\lambda
)_{\lambda\in\Lambda } \in \ell ^2(\Lambda )$ there exists $f\in \cF
^2$, such that $f(\lambda ) e^{-\pi |\lambda |^2/2} = a_\lambda $.
Equivalently, the set of normalized reproducing kernels
$\kappa _\lambda = k_\lambda /  \|k_\lambda \|_{\fc} = k_\lambda
/ k(\lambda,\lambda )^{1/2}$ is a Riesz
sequence, i.e., there exists $A,B>0$, such that
\begin{equation}
  \label{eq:re1a}
 A \|a\|_2^2  \leq \| \sum _{\lambda \in \Lambda } a_\lambda  \kappa _\lambda \|_{\fc
}^2  \leq B \|a\|_2^2
\end{equation}
for all $a\in \ell ^2(\Lambda )$. It suffices to  require
\eqref{eq:re1a} only for all $a$ with finite support.

In analogy to \mz\ families for sampling, we define uniform families for
interpolation as follows. We denote the normalized reproducing kernels
in $\pln $ by $\kappa _{n,\lambda} = k_{n,\lambda } / \|k_{n,\lambda
}\|_{\fc} $. 
\begin{definition}
  A sequence of finite sets $\Lambda _n\subseteq \bC $ is a uniform interpolating family for $\pln $ in $\cF ^2 $,
  if there exist  constants $A,B >0$ independent of $n$, such that for
  $n$ large enough, $n\geq n_0$, 
  \begin{equation}
    \label{eq:re2}
A \|a\|_2^2 \leq     \| \sum _{\lambda \in \Lambda _n } a_\lambda  \kappa _{n,\lambda} \|_{\fc
}^2 \leq B \|a\|_2^2   \qquad \text{ for all } a\in \ell ^2(\Lambda _n)
\, .
\end{equation}
\end{definition}
Equivalently, for every $a\in \ell ^2(\Lambda _n)$ there exists a
polynomial $p\in \pln $, such that
\begin{align*}
\frac{p(\lambda )}{k_n(\lambda,\lambda)^{1/2}} = a_\lambda \qquad
                                                 \text{ and } \qquad 
\|p\|_{\fc } ^2  \leq A \|a\|_2^2 \, .
\end{align*}
A further equivalent condition is that the associated Gram matrix
with entries $G_{\mu ,\lambda} = \langle \kappa _{n,\lambda},
\kappa _{n,\mu }\rangle$ has the smallest eigenvalues $\lambda
_{min} \geq A$~\cite[\S 2.3 Lem.~2]{meyer92}.

The relation between sets of interpolation for $\cF ^2 $ and uniform
interpolating families is similar to the case of sampling.

\begin{tm}\label{mzint}
    Assume that $\Lambda \subseteq \bC $ is a  set of interpolation for $\cF ^2
    $.  For $\tau >0$ and $\epsilon >0$ set $\rho _n$, such that $\pi \rho _n^2 =
  n-\sqrt{n} (\sqrt{2\log n} + \tau) $. Then for every   $\tau >0$ large 
enough,   the sets $\Lambda _n = 
  \Lambda \cap  B_{\rho_n} $  form a uniform interpolating family for $\pln $ in
  $\cF ^2$.
\end{tm}
\begin{proof}

  Since $D^+(\Lambda ) <1$ is necessary for an interpolating set in
  $\cF ^2$ by ~\cite{seip92}, the definition of $\rho _n$ implies that
  $$
  \# (\Lambda \cap B_{\rho _n}) \leq 1 \cdot |B_{\rho _n}| \leq n
  $$
  for $n\geq n_0$ large enough. Consequently $\Lambda _n = \Lambda
  \cap B_{\rho _n}$ contains at most $n$ points. 
  
  We show that we can choose $\tau >0$ in such a manner that, for
  $a\in \ell ^2(\Lambda )$ with finite support and all $n\in \bN $ sufficiently large, 
  \begin{equation}
    \label{eq:re3}
\|\sum _{\lambda \in \Lambda _n} a_\lambda (\kappa _\lambda -
\kappa _{n,\lambda })\|_{\fc } ^2 \leq \frac{A}{4} \|a\|_2^2 \, .    
  \end{equation}
Then via the triangle inequality  $\tfrac{A}{4} \|a\|_2^2 \leq  \|\sum _{\lambda \in \Lambda _n} a_\lambda 
\kappa _{n,\lambda }\|_{\fc } ^2 \leq (B+ \tfrac{A}{4})  \|a\|_2^2 $. . 

Denote the difference of the kernels by $e_\lambda = \kappa _\lambda -
\kappa _{n,\lambda }$ and the Gram matrix of $e_\lambda $ by $E$
with entries $E_{\lambda , \mu } = \langle e_\mu , e_\lambda \rangle,
\lambda , \mu \in \Lambda _n$. Then \eqref{eq:re3} amounts to saying
the $\|E\|_{\mathrm{op}} \leq A/4$. Since $E$ is positive
(semi-)definite, it suffices to bound the trace of $E$. To do this,
consider the diagonal elements of $E$ first. We see that
\begin{align*}
E_{\lambda ,\lambda  } &= \|\kappa _\lambda -
\kappa _{n,\lambda }\|_{\fc } ^2 \notag \\ 
&= 2 - 2 \mathrm{Re}\, \langle \kappa_\lambda,
\kappa _{n,\lambda}\rangle \notag  \\
&= 2\Big(1-\frac{\langle k_\lambda , k_{n,\lambda }\rangle}{k(\lambda,\lambda )^{1/2} 
k_n(\lambda,\lambda )^{1/2}}\Big) \notag \\
&= 2\Big(1-\frac{ k_n (\lambda,\lambda)^{1/2}}{k(\lambda,\lambda )^{1/2} }\Big)
\, .  \notag 
\end{align*}
Since $k_n(\lambda,\lambda ) < k(\lambda,\lambda)$, the estimate for
the diagonal elements simplifies to 
$$
E_{\lambda,\lambda} \leq 2\Big(1-\frac{ k_{n
  }(\lambda,\lambda)}{k(\lambda,\lambda )}\Big) = 2\Big( 1
-\frac{\Gamma (n+1,\pi |\lambda |^2)}{n!}\Big)\, .
$$
If $x\leq n-\sqrt{n} \tau_n^2$ (with $\tau _n$ depending on $n$), then by 
Proposition~\ref{asympt1}(iii). 
$$
1-\frac{\Gamma(n+1,x)}{n!} \leq 1-\frac{\Gamma(n+1,n-\sqrt{n}\tau
  _n)}{n!} \leq  e^{-\tau _n^2/2}$$ 

Combining these observations, we arrive at 
\begin{align*}
  \|E\|_{\mathrm{op}} &\leq 2 \sum _{\lambda \in \Lambda \cap B_{\rho _n}} \Big( 1
-\frac{\Gamma (n+1,\pi |\lambda |^2)}{n!}\Big)\\
&\leq 2 n e^{-\tau _n^2/2}
\end{align*}
By choosing $\tau _n = \sqrt{2\log n} + \tau $, with $\tau >0$
 large enough,
we achieve $\|E\|_{\mathrm{op}} \leq A/4$ for $n\geq n_0$. As we have
seen, this suffices to conclude that $\kappa _{n,\lambda}$ is a Riesz
  sequence in $\pln$ with lower constant independent of the degree
  $n$. 
\end{proof}

Similar to the case of \mz\ families for sampling, we obtain uniform
families for interpolation with the correct cardinality.

\begin{cor} \label{mzintcount}
    For every $\epsilon >0$ there exist uniform interpolating families 
$(\Lambda _n)$ for
  $\pln $ in $\cF ^2$ with $\# \Lambda _n \geq (1-\epsilon ) (n+1)$
  points. 
\end{cor}
\begin{proof}
  The proof is similar to the one of Corollary~\ref{mzsampcount}. 
\end{proof}

\begin{tm}\label{mztoint}
  Assume that
  $(\Lambda _n)$ is a uniform interpolating family for the polynomials $\pln $ 
in
  $\cF ^2 $. Let $\Lambda $ be a weak limit of $(\Lambda
  _n)$ or of  some subsequence $(\Lambda _{n_k})$.
  Then $\Lambda $ is a set of interpolation for $\cF ^2 $. 
\end{tm}
\begin{proof}
 Let $\Lambda $ be a weak limit of $\Lambda _n$ (or some
 subsequence), and let $a\in \ell ^2(\Lambda )$ with finite support in
 some disk $B_{\rho_N}$ say. Enumerate $\Lambda \cap B_{\rho _N} =
 \{\lambda _j: j=1, \dots , L\}$. By weak convergence,  for every
 $\lambda _j \in \Lambda \cap
 B_{\rho _N}$ there is a sequence $\lambda ^{(n)}_j\in \Lambda _n$, such
 that $\lim _{n\to \infty } \lambda _j^{(n)} = \lambda _j $.

 We  show that 
  \begin{equation}
    \label{eq:re6}
\lim _{n\to \infty }   \|\sum _{j=1}^L a_{\lambda_j} (\kappa _{\lambda _j} -
\kappa _{n,\lambda _j^{(n)} })\|_{\fc } ^2 = 0\, .    
  \end{equation}
 Consequently,
 $$
\| \sum _{\lambda \in \Lambda \cap B_{\rho _N}} a_\lambda
  \kappa _\lambda \|_{\fc } = \lim _{n\to \infty } \| \sum _{j=1}^L
  a_{\lambda _j^{(n)}}   \kappa _{n,\lambda_j^{(n)}} \|_{\fc } \geq A \|a\|_2^2
  \, ,
  $$
  because $(\Lambda _n)$ is a uniform interpolating family. Thus
  $\{\kappa _\lambda:\lambda \in \Lambda \}$ is a Riesz sequence in
  $\cF ^2 $.

  To show \eqref{eq:re6},  we set $e_j = \kappa _{\lambda _j} -
\kappa _{n,\lambda _j^{(n)} }$ and consider the associated Gramian
with entries $E_{jk} = \langle e_k,e_j\rangle$. Again we use
\begin{align*}
\|E\| _{\mathrm{op}} & \leq \mathrm{tr}\, E = \sum _{j=1}^L
\|e_j\|^2_{\fc } = \sum _{j=1}^L \|\kappa _{\lambda _j} -
  \kappa _{n,\lambda _j^{(n)} } \|^2 _{\fc } \\
&= 2 \sum _{j=1}^L   \Big(1-\mathrm{Re}\, \langle \kappa _{\lambda _j} ,
    \kappa _{n,\lambda _j^{(n)} }\rangle\Big) \, .
\end{align*}
 Consider a single term of this sum and write  $\lambda _j^{(n)} = \lambda $ and $\lambda _j= \mu
 $ for fixed $j$. Note that $|\lambda - \mu | \leq 1$ for $n$ large enough and that
 $\pi |\mu |^2 \leq N$ by the assumption that $\supp a \subseteq
 B_{\rho _N}$.
 Now
 \begin{align*}
\mathrm{Re} \,  \langle \kappa _{\lambda _j} ,
    \kappa _{n,\lambda _j^{(n)} } \rangle & = \mathrm{Re} \,\frac{k_n(\lambda,\mu )}{k_n(\lambda,\lambda)^{1/2} k(\mu,\mu)^{1/2}} \\   
&= e^{-\pi |\lambda -\mu |^2/2} \frac{\Gamma (n+1,\pi |\lambda |^2) +
\big[ \Gamma (n+1,\pi \lambda \bar{\mu}) -\Gamma (n+1,\pi |\lambda  |^2)\big]}
{ \Gamma(n+1,\pi |\lambda|^2)^{1/2} \, n!^{1/2}} \\
&= e^{-\pi |\lambda -\mu |^2/2} \frac{\Gamma (n+1,\pi |\lambda |^2)^{1/2}}{
 n!^{1/2}} + e(n,\lambda,\mu)\, .
 \end{align*}
Since 
$$
\frac{ \Gamma (n+1,\pi \lambda \bar{\mu}) -\Gamma (n+1,\pi |\lambda
  |^2))}{n!} \leq e^{-n\eta } \, 
$$
for some $\eta >0$ by Proposition~\ref{offd} and $\frac{\Gamma (n+1, \pi |\lambda |^2)}{n!}
\to 1$, the  term $e(n,\lambda ,\mu )$ tends to zero, as $n\to \infty
$. By a similar reasoning, as $n\to \infty $ and thus $\lambda =
\lambda _j^{(n)} \to \lambda _j= \mu $, we have
$$
1 - e^{-\pi |\lambda -\mu |^2/2} \frac{\Gamma (n+1,\pi |\lambda |^2)^{1/2}}{
  n!^{1/2}} \to 0
$$
for finitely many terms. Unraveling the notation, this means that
$\mathrm{tr} \, E \to 0$ and \eqref{eq:re6} is proved. 
\end{proof}

\begin{prop}
  There is no \mz\ family $(\Lambda _n)$ for $\pln $ in $\cF ^2$ with
  $\# \Lambda _n = n+1$.
\end{prop}
\begin{proof}
A set $\Lambda _n$ with $n+1$ points is both sampling and interpolating for $\pln $
with the same constants for  interpolation as for sampling. By
Theorem~\ref{mzsampa} 
any weak limit $\Lambda $ of a \mz\ family is a sampling set for $\cF ^2$, and
by Theorem~\ref{mztoint} $\Lambda $ is a set of interpolation for $\cF
^2$. This is a contradiction, since $\cF ^2$ does not admit any sets
that are simultaneously sampling and interpolating. See, e.g.,~\cite[Lemma~6.2]{seip92}. 
\end{proof}

\section{Gabor frames for subspaces spanned by Hermite functions}

By using the well-known connection between sampling in Fock space and
the theory of Gaussian Gabor frames we may rephrase the main results
 in the language of Gabor frames for subspaces.

Recall that the Bargman transform is defined to be
$$
Bf(z) = 2^{1/4} \int _{\bR } f(t) e^{2\pi zt -\pi t^2} \, dt \,\, e^{-\pi z^2/2}
$$
for  $z\in \bC $. It maps functions and distributions on $\bR $ to
entire functions.

We use the following properties of the Bargman transform. See e.g.,~\cite{folland89}.

(i) The  Bargman transform is unitary from $L^2(\bR )$ onto Fock space
$\cF ^2 $.

(ii) Let $\phi _z(t) = e^{-2\pi i y t} e^{-\pi (t-x)^2}$ denote the
\tfs\ of the Gaussian by $z=x+iy$. Then
$$
B\phi _z(w) = k_z(w) = e^{\pi \bar{z} w}
$$
is the reproducing kernel of $\cF ^2 $.

(iii) $B$ maps the normalized Hermite functions $h_k$, $$h_k(t) = c_k
e^{\pi t^2} \tfrac{d^k}{dt^k}(e^{-2\pi t^2}), \quad \|h_k\|_2=1, $$ to the
monomials $
 e_k(z) = \big( \frac{\pi ^k}{k!}\big)^{1/2}  \, z^k
 $. 
With the Bargman transform all questions about the spanning properties
of \tfs s $\phi _z$ of the Gaussian can be translated into questions
about the reproducing kernels $k_z$ in Fock space. For instance,
$\{\phi _\lambda : \lambda \in \Lambda \}$ is a frame for $L^2(\bR )$, \fif\ $\Lambda
$ is a sampling set for $\cF ^2$. Almost all statements about
Gaussian Gabor frames have been obtained via complex analysis methods,
notably the complete characterization of Gaussian Gabor frames  by
Lyubarski~\cite{lyub92} and Seip~\cite{seip92} and many subsequent
detailed investigations~\cite{ALS09,BBK20}.  
To this line of thought we add a statement about Gabor frames for
distinguished subspaces spanned by Hermite polynomials. Constructions
of this type have been used in signal processing~\cite{Hla98}. 

    \begin{tm}
      Assume that $\Lambda $ is a sampling set for $\cF ^2$, or
      equivalently  $\cG (h_0, \Lambda) = \{\phi _\lambda : \lambda\in \Lambda\} $ is 
Gabor frame in $L^2(\bR )$, then $\{\phi _\lambda : \pi |\lambda | ^2 \leq n+\sqrt{n}\tau \}$ is a
      frame for $V_n = \mathrm{span}\, \{h_k : k=0, \dots , n\}$ with
      bounds independent of $n$, i.e.,
      $$
   A\|f\|_2 \leq \sum _{\lambda \in \Lambda : \pi |\lambda |^2 \leq
     n+\sqrt{n}\tau } |\langle f, \phi _\lambda  \rangle |^2 \leq
   B\|f\|_2^2 \qquad \text{ for all } f\in V_n \, .
   $$
    \end{tm}

    \begin{proof}
The statement is equivalent to Theorem~\ref{samptomz} via the Bargman
transform.\footnote{Since $\phi _\lambda \not\in V_n$, some authors
  use the term ``pseudoframe'' for this situation.} 
    \end{proof}

\section{Appendix}

For completeness we present some elementary estimates for the zero
order asymptotics of the incomplete Gamma function that imply
Corollary~\ref{asympt1}.
Using Stirling's formula for $\big(\tfrac{n}{e}\big)^n
\sqrt{2\pi n}\leq n! \leq e \sqrt{n} \big(\tfrac{n}{e}\big)^n$,  we write
\begin{equation*}
  \frac{\Gamma (n+1, n+\sqrt{n}\tau )}{n!} \asymp \frac{1}{n!}\int
  _{n+\sqrt{n}\tau } ^\infty t^n e^{-t} \, dt \asymp
  \frac{1}{\sqrt{n}} \int
  _{n+\sqrt{n}\tau } ^\infty \Big(\frac{t}{n}\Big)^n e^{n-t} \, dt \, .
\end{equation*}
(The constants in the equivalence are in $[e\inv ,(2\pi )^{-1/2}]$.) 
Using the substitution $u=\frac{t}{n}-1$ we obtain
\begin{align}
    \frac{\Gamma (n+1, n+\sqrt{n}\tau )}{n!} &\asymp \sqrt{n} \int
    _{\tau / \sqrt{n}} ^\infty (1+u)^n e^{-nu}\, du \notag \\
  &= \sqrt{n} \int
    _{\tau / \sqrt{n}} ^\infty e^{ -n ( u - \ln (1+u))}\, du =
    \sqrt{n} \Big(  \int
    _{\tau / \sqrt{n}} ^1 \dots + \int _1^\infty \dots dt \Big) \, .  \label{incom2}
  \end{align}
  Using the 
inequality   $u-\ln (1+u)
  \geq (1-\ln 2) u$ valid for  $u\in [1,\infty )$, the latter
  integral is bounded by
  \begin{equation}
    \label{eq:eq:369}
  \sqrt{n} \int
    _1 ^\infty e^{ -n ( u - \ln (1+u))}\, du \leq \sqrt{n} \int
    _1 ^\infty e^{ -n (1-\ln 2)u}\, du \leq
  \frac{\sqrt{n}}{n(1-\ln 2)} = \cO \big(\frac{1}{\sqrt{n}} \big) \,  .
    \end{equation}
    In the first integral in \eqref{incom2} we use the power series of
  $\ln $ and obtain, for $u\in [0,1]$, 
  $$
  u-\ln (1+u) = \sum _{k=2}^\infty \frac{(-1)^k}{k} u^k \geq
  \frac{u^2}{2} - \frac{u^3}{3} \geq \frac{u^2}{6} \, .
  $$ Consequently,
  \begin{align*}
    \sqrt{n} \int _{\tau / \sqrt{n}} ^1 (1+u)^n e^{-nu}\, du &\leq
\sqrt{n} \int _{\tau / \sqrt{n}} ^1 e^{-n u^2/6} \, du \\
&= \sqrt{n} \sqrt{\frac{6}{n}} \int _{\tau / \sqrt{6}} ^{\sqrt{n/6}}
 e^{-v^2} \lesssim \mathrm{erf} (\tau / \sqrt{6}) \leq e^{-\tau ^2/6}
\, . 
  \end{align*}
  From $u-\ln (1+u) \leq u^2/2$ we obtain the lower bound
  \begin{align*}
   \sqrt{n} \int _{\tau / \sqrt{n}} ^1 (1+u)^n e^{-nu}\, du &\geq
\sqrt{n} \int _{\tau / \sqrt{n}} ^1 e^{-n u^2/2} \, du \\
&= \sqrt{n} \sqrt{\frac{2}{n}} \int _{\tau / \sqrt{2}} ^{\sqrt{n/2}}
 e^{-v^2}  \gtrsim  \mathrm{erf} (\tau / \sqrt{2})   \gtrsim
\frac{1}{\tau} e^{-\tau ^2/2} \, ,
  \end{align*}
  for $n$ large enough. Items (i) and (ii) of Proposition~\ref{asympt1}
  now follow easily from
  these  estimates. 

A weaker version of   item (v) follows by setting $\tau =0$ in the
above estimates. 
  
  \vspace{3mm}

(iii)   Similarly, 
\begin{align*}
 1-  \frac{\Gamma (n+1, n-\sqrt{n}\tau )}{n!} &\leq 
 \frac{1}{\sqrt{2\pi n}} \Big(\frac{e}{n} \Big)^n \int _0
  ^{n-\sqrt{n}\tau } t^n e^{-t} \, dt =\\
  &=\frac{1}{\sqrt{2\pi n}} \int _0
  ^{n-\sqrt{n}\tau } \Big(\frac{t}{n}\Big)^n e^{n-t} \, dt \, .
\end{align*}
With  the substitution $u=1-\frac{t}{n}$ we obtain
\begin{align}
&\frac{1}{\sqrt{2\pi n} } \int _0  ^{n-\sqrt{n}\tau } \Big(\frac{t}{n}\Big)^n 
e^{n-t} \, dt
= \frac{\sqrt{n}}{\sqrt{2\pi}} \int _{\tau/\sqrt{n}}^1 (1-u)^n e^{nu}\, du = 
\notag \\
&  = \frac{\sqrt{n}}{\sqrt{2\pi}} \int _{\tau/\sqrt{n}}
^1 e^{ n ( u +\ln (1-u))}\, du =
\frac{\sqrt{n}}{\sqrt{2\pi}} \Big(  \int   _{\tau/\sqrt{n}}
^{1/2} \dots  + \int _{1/2} ^1 \dots dt \Big)  \, .
\end{align}

Since $u+\ln (1-u) \leq 1/2 - \ln 2 < 0$ on the interval $[1/2,1]$,
the second term decays exponentially in $n$. For
the first term we use $u+\log (1-u) = - \sum _{k=2}^\infty \frac{u^k}{k}
\leq - u^2/2 $ and obtain
\begin{align*}
\frac{\sqrt{n}}{\sqrt{2\pi}} \int _{\tau/\sqrt{n}}
^{1/2} e^{ n ( u + \ln (1-u))}\, du &\leq \frac{\sqrt{n}}{\sqrt{2\pi}}
                                      \int _{\tau/\sqrt{n}}
^{1/2} e^{ -n u^2/2}\, du \\
&=  \frac{1}{\sqrt{2\pi}}\int _{\tau}
^{\sqrt{n}/2} e^{ -v^2/2}\, du \leq \tfrac{1}{2} e^{-\tau ^2/2} \, .  
\end{align*}
(iv) follows from (iii). If $n-\sqrt{n}\tau \geq x$, then
$$
\frac{\Gamma (n+1,x)}{n!} \geq \frac{\Gamma (n+1,n-\sqrt{n}\tau  )}{n!} \geq 1 - e^{-\tau
  ^2} \, .
$$
Since this holds arbitrary $\tau $ and $n$ large, we obtain $\lim
_{n\to \infty } \frac{\Gamma (n+1,x)}{n!} = 1$. 


\def\cprime{$'$} \def\cprime{$'$} \def\cprime{$'$} \def\cprime{$'$}
  \def\cprime{$'$} \def\cprime{$'$}

\def\cprime{$'$} \def\cprime{$'$} \def\cprime{$'$} \def\cprime{$'$}
  \def\cprime{$'$} \def\cprime{$'$}


\begin{thebibliography}{10}

\bibitem{Ameur13}
Y.~Ameur.
\newblock Near-boundary asymptotics for correlation kernels.
\newblock {\em J. Geom. Anal.}, 23(1):73--95, 2013.

\bibitem{AHM10}
Y.~Ameur, H.~K. Hedenmalm, and N.~Makarov.
\newblock Berezin transform in polynomial {B}ergman spaces.
\newblock {\em Comm. Pure Appl. Math.}, 63(12):1533--1584, 2010.

\bibitem{AOC12}
Y.~Ameur and J.~Ortega-Cerd\`a.
\newblock Beurling-{L}andau densities of weighted {F}ekete sets and correlation
  kernel estimates.
\newblock {\em J. Funct. Anal.}, 263(7):1825--1861, 2012.

\bibitem{AR20}
Y.~Ameur and J.~L. Romero.
\newblock The planar low temperature {C}oulomb gas: separation and
  equidistribution.
\newblock arXiv:2010.10179.

\bibitem{AC21}
Y.~Ameur and J.~Cronvall.
\newblock Szegő type asymptotics for the reproducing kernel in spaces of
  full-plane weighted polynomials.
\newblock 2021.
\newblock Preprint, https://arxiv.org/abs/2107.11148.


\bibitem{ALS09}
G.~Ascensi, Y.~Lyubarskii, and K.~Seip.
\newblock Phase space distribution of {G}abor expansions.
\newblock {\em Appl. Comput. Harmon. Anal.}, 26(2):277--282, 2009.

\bibitem{BBK20}
Y.~Belov, A.~Borichev, and A.~Kuznetsov.
\newblock Upper and lower densities of {G}abor {G}aussian systems.
\newblock {\em Appl. Comput. Harmon. Anal.}, 49(2):438--450, 2020.

\bibitem{folland89}
G.~B. Folland.
\newblock {\em Harmonic Analysis in Phase Space}.
\newblock Princeton Univ. Press, Princeton, NJ, 1989.

\bibitem{GST12}
A.~Gil, J.~Segura, and N.~M. Temme.
\newblock Efficient and accurate algorithms for the computation and inversion
  of the incomplete Gamma function ratios.
\newblock {\em SIAM J. Sci. Comput.}, 34(6):A2965--A2981, 2012.

\bibitem{Gro99}
K.~Gr{\"o}chenig.
\newblock Irregular sampling, {T}oeplitz matrices, and the approximation of
  entire functions of exponential type.
\newblock {\em Math. Comp.}, 68(226):749--765, 1999.

\bibitem{Gro20}
K.~Gr\"{o}chenig.
\newblock Sampling, {M}arcinkiewicz-{Z}ygmund inequalities, approximation, and
  quadrature rules.
\newblock {\em J. Approx. Theory}, 257:105455, 20, 2020.

\bibitem{GOC21}
K.~Gr{\"o}chenig and J.~Ortega-Cerd{\`a}.
\newblock Marcinkiewicz-Zygmund inequalities for polynomials in Bergman and
  Hardy spaces.
\newblock {\em J. Geom. Anal.},  J. Geom. Anal. 31(7):7595--7619,
2021. 

\bibitem{GOR15}
K.~Gr{\"o}chenig, J.~Ortega-Cerd{\`a}, and J.~L. Romero.
\newblock Deformation of {G}abor systems.
\newblock {\em Adv. Math.}, 277:388--425, 2015.

\bibitem{Hla98}
  F. Hlawatsch.
  \newblock Time-frequency Analysis and Synthesis of Linear Signal
  Spaces: Time-Frequency Filters Signal Detection and Estimation, and
  Range-Doppler Estimation.
  \newblock Kluwer Academic Publishers, Boston, 1998. 


\bibitem{KKLT}
B.~Kashin, E.~Kosov, I.~Limonova, V.~Temlyakov.
\newblock Sampling discretization and related problems.
\newblock Preprint. https://arxiv.org/pdf/2109.07567.pdf. 



\bibitem{LOC16}
N.~Lev and J.~Ortega-Cerd\`a.
\newblock Equidistribution estimates for {F}ekete points on complex manifolds.
\newblock {\em J. Eur. Math. Soc. (JEMS)}, 18(2):425--464, 2016.

\bibitem{lyub92}
Y.~I. Lyubarski{\u\i}.
\newblock Frames in the {B}argmann space of entire functions.
\newblock In {\em Entire and subharmonic functions}, pages 167--180. Amer.
  Math. Soc., Providence, RI, 1992.
  
\bibitem{meyer92}
Y.~Meyer.
\newblock {\em Wavelets and operators.}
\newblock Cambridge University Press, 1992.

\bibitem{Nem16}
G.~Nemes.
\newblock The resurgence properties of the incomplete Gamma function, {I}.
\newblock {\em Anal. Appl. (Singap.)}, 14(5):631--677, 2016.

\bibitem{Nem16a}
G.~Nemes.
\newblock The resurgence properties of the incomplete Gamma function, {I}.
\newblock {\em Anal. Appl. (Singap.)}, 14(5):631--677, 2016.

\bibitem{NIST}
{\it NIST Digital Library of Mathematical Functions}.
\newblock http://dlmf.nist.gov/, Release 1.1.2 of 2021-06-15.
\newblock F.~W.~J. Olver, A.~B. {Olde Daalhuis}, D.~W. Lozier, B.~I. Schneider,
  R.~F. Boisvert, C.~W. Clark, B.~R. Miller, B.~V. Saunders, H.~S. Cohl, and
  M.~A. McClain, eds.

\bibitem{OS07}
J.~Ortega-Cerd\`a and J.~Saludes.
\newblock Marcinkiewicz-{Z}ygmund inequalities.
\newblock {\em J. Approx. Theory}, 145(2):237--252, 2007.

\bibitem{seip92}
K.~Seip.
\newblock Density theorems for sampling and interpolation in the
  {B}argmann-{F}ock space. {I}.
\newblock {\em J. Reine Angew. Math.}, 429:91--106, 1992.

\bibitem{seip04}
K.~Seip.
\newblock {\em Interpolation and sampling in spaces of analytic functions},
  volume~33 of {\em University Lecture Series}.
\newblock American Mathematical Society, Providence, RI, 2004.

\bibitem{seip-wallsten}
K.~Seip and R.~Wallst{\'e}n.
\newblock Density theorems for sampling and interpolation in the
  {B}argmann-{F}ock space. {I}{I}.
\newblock {\em J. Reine Angew. Math.}, 429:107--113, 1992.

\bibitem{simon2}
B.~Simon.
\newblock {\em Basic complex analysis}.
\newblock A Comprehensive Course in Analysis, Part 2A. American Mathematical
  Society, Providence, RI, 2015.

\bibitem{Tem75}
N.~M. Temme.
\newblock Uniform asymptotic expansions of the incomplete Gamma functions and
  the incomplete beta function.
\newblock {\em Math. Comp.}, 29(132):1109--1114, 1975.

\bibitem{Tri50}
F.~G. Tricomi.
\newblock Asymptotische {E}igenschaften der unvollst\"{a}ndigen
  {G}ammafunktion.
\newblock {\em Math. Z.}, 53:136--148, 1950.

\end{thebibliography}
\end{document}